\documentclass[10pt,a4paper]{amsart}
\usepackage{amsmath,amssymb,amscd,amsthm,verbatim}
\usepackage[all]{xy}
\usepackage{enumerate}

\newenvironment{enumeratei}{\begin{enumerate}[\upshape (i)]}%
{\end{enumerate}}

\theoremstyle{plain}
\newtheorem{theorem}{Theorem}
\newtheorem{corollary}[theorem]{Corollary}
\newtheorem{lemma}[theorem]{Lemma}
\newtheorem{proposition}[theorem]{Proposition}

\numberwithin{equation}{section}
\numberwithin{theorem}{section}
 \allowdisplaybreaks
\title[Free centre-by-nilpotent-by-abelian Lie rings]
{Torsion in free centre-by-nilpotent-by-abelian Lie rings  of rank 2 }

\author{Ralph St\"{o}hr}
\address{School of Mathematics, University of Manchester, Alan Turing Building,
Manchester, M13 9PL,
United Kingdom}\email{Ralph.Stohr@manchester.ac.uk }

\date{\today}


\begin{document}

\begin{abstract}
For $c\geq 2$,
the free centre-by-(nilpotent-of-class-c-1)-by abelian Lie ring on a set $X$ is the quotient
$L/[(L')^c,L]$ where $L$
is the free Lie ring on $X$, and $(L')^c$ denotes the $c$th term of the lower central series of the derived ideal $L'=L^2$ of $L$. In this paper we give a complete description of the torsion subgroup of its additive group in the case where $|X|=2$ and $c$ is a prime number.
\end{abstract}

\maketitle

\section{Introduction}

For an integer  $c\geq 2$, the free centre-by-(nilpotent-of-class $c-1$)-by-abelian Lie ring on a set $X$ is the quotient
    \begin{equation}\label{1.1}
   L /[(L')^c,L]
\end{equation}
 where $L$ is an (absolutely) free Lie ring on $X$ and  $(L')^c$ is the $c$th term of the lower central series of the derived ideal $L'=L^2$ of $L$. In this note we determine the torsion subgroup of the additive group of \eqref{1.1} in the case where $L$ has rank $2$, that is $X$ is a set of two elements, and $c$ is a prime number. It is easily observed that for any $c\geq2$ and any $X$ with $|X|\geq2$ the torsion subgroup of \eqref{1.1} is contained in the central ideal  $(L')^c /[(L')^c,L]$. If $c=p$, where $p$ is a prime, then the ideal $(L')^p /[(L')^p,L]$ is a direct sum of a free abelian group and an infinite elementary abelian $p$-group. We exhibit an explicit
 $\mathbb{Z}_p$-basis of the torsion part if $|X|=2$. This is the main result of this paper, see Theorem 5.1 in Section 5. In the earlier paper \cite{alexandroustohr} a complete description of the torsion subgroup of the additive group of the more reduced quotient
 \begin{equation}\label{1.2}
    L/([(L')^c,L]+L''')
 \end{equation}
 was obtained, again in the case of rank $2$, but for arbitrary $c\geq 2$. Note that \eqref{1.2} coincides with \eqref{1.1} for $c=2$ and $c=3$, but not for $c\geq 4$. We make essential use of the results from \cite{alexandroustohr}.

 The interest in torsion in  relatively free Lie rings of the form \eqref{1.1} as well as their group-theoretic counterparts, the relatively free groups $F/[\gamma_c(F'),F]$,where $F$ is a free group and $\gamma_c(F')$ is the $c$th term of the lower central series of the derived subgroup $F'$, has a long history. If $c=2$, these turn into the free centre-by-metabelian Lie rings and the free centre-by-metabelian groups, respectively, and it was the latter in which Kanta Gupta \cite{gupta} first discovered torsion elements, a major surprise at the time. Later torsion was detected in $F/[\gamma_c(F'),F]$ if $c$ is a prime \cite{stohr87} and if $c=4$ \cite{stohr93}. Quite surprising, though, it turned out that $F/[\gamma_c(F'),F]$ is torsion free if $c$ is divisible two distint primes  \cite{johnsonstohr}. As to Lie rings, early work in \cite{kuz'min77} on the free centre-by-metabelian Lie rings $L/[L'',L]$, and in particular on torsion in the additive group, turned out in need of some rectification, and this was eventually accomplished in \cite{mansuroglustohr} and \cite{kovacsstohr}. For larger values of $c$,  Drensky \cite{drensky} proved that for any prime $p$ the free Lie ring $L/[(L')^p,L]$ contains non-trivial multilinear elements of degree $2p+1$ which have order $p$.

 The paper is organized as follows. In Section 2 we introduce notation and assemble a number of preliminary results on the homogeneous components of free Lie rings. These are further examined in Section 3 where we derive our main result on Lie powers of free modules in prime degree, postponing, however, one key ingredient on Lie powers of prime degree $p$ over fields of characteristic $p$ to Section 4. In the final Section 5 we exploit the results of the previous sections to derive our main result.

\section{Notation and some preliminaries}

We write maps on the right and use left-normed notation for Lie products. Let $A$ be a free abelian group. By $L(A)$ we denote the free Lie ring on $A$. For a positive integer $c$, we let $L^c(A)$ denote the degree $c$ homogeneous component of $L(A)$, that is the span of all left normed simple Lie products  $[a_1,a_2,\dots,a_c]$ with $a_i\in A$. In particular, $L^1(A)=A$. The universal envelope of $L(A)$ can be identified with the tensor ring $T(A)=\bigoplus_{c\geq 0}T^i(A)$ where $T^0(A)=\mathbb{Z}$ and, for $c>0$, $T^c(A)=\underbrace{A\otimes\dots\otimes A}_c$, the $c$th tensor power of $A$. By $A^c$ we denote the $c$th symmetric power of $A$. The free metabelian Lie ring on $A$ is the quotient of $L(A)$ by its second derived ideal: $M(A)=L(A)/L(A)''$. This too is a graded Lie ring and we let $M^c(A)$ denote its $c$th homogeneous component, that is $M^c(A)=L^c(A)/(L^c(A)\cap L(A)'')$. We call $L^c(A)$ and $M^c(A)$ the $c$th free Lie power and the $c$th free metabelian Lie power of $A$, respectively. It is well-known that if $\mathcal{A}$ is an ordered $\mathbb{Z}$-basis of $A$, then the left normed simple Lie products $[b_1,b_2,\dots,b_c]$ with $b_i\in \mathcal{A}$ and $b_1>b_2\leq\dots\leq b_c$ form a $\mathbb{Z}$-basis of $M^c(A)$ (see \cite[Section 4.2.2]{bakhturin}).

The canonical embedding of $L(A)$ into its universal envelope $T(A)$ induces in each degree $c$ an embedding $\nu_c:L^c(A)\rightarrow T^c(A)$. By a well-known theorem of Wever (see \cite[Chapter 5, Theorem 5.16]{MKS}), the composite of this embedding with the natural projection $\rho_c:T^c(A)\rightarrow L^c(A)$ defined by $a_1\otimes\dots\otimes a_c\mapsto [a_1,\dots,a_c]$ amounts to multiplication by $c$ on $L^c(A)$:
\begin{equation}\label{2.1}
    L^c(A) \xrightarrow{\nu_c} T^c(A)\xrightarrow{\rho_c} L^c(A),\hspace{2cm} \nu_c\rho_c=c.
\end{equation}
 The definition of $M^c(A)$ gives rise to a short exact sequence
\begin{equation}\label{2.2}
    0\rightarrow B^c(A)\rightarrow L^c(A)\xrightarrow{\eta_c} M^c(A)\mapsto 0
\end{equation}
where $B^c(A)=L^c(A)\cap L(A)''$ and $\eta_c$ is the natural projection map. Moreover, for $c\geqslant 2$ the metabelian Lie power  $M^c(A)$ fits into a short exact sequence
 \begin{equation}\label{2.3}
   0\rightarrow M^c(A)\xrightarrow{\mu_c} A\otimes A^{c-1} \xrightarrow{\kappa_c} A^c \rightarrow 0
 \end{equation} where the maps $\mu_c$ and $\kappa_c$ are given by \begin{equation*}
   [a_1,a_2,\dots,a_c]\mapsto a_1\otimes(a_2\circ\dots\circ a_c)-a_2\otimes(a_1\circ\dots\circ a_c)
 \end{equation*}
 and $a_1\otimes(a_2\circ\dots\circ a_c)\mapsto a_1\circ a_2\circ\dots\circ a_c$, respectively (see \cite[Corollary 3.2]{hannebauerstohr}). Moreover, there is a map $\lambda_c:A\otimes A^{c-1}\rightarrow M^c(A)$, given by
 \begin{equation*}
   a_1\otimes(a_2\circ\dots\circ a_c)\mapsto [a_1,a_2,a_3,\dots,a_c]+[a_1,a_3,a_2,\dots,a_c]+\dots+[a_1,a_c,a_2,\dots,a_{c-1}],
 \end{equation*}
 such that the composite of $\mu_c$ and $\lambda_c$ amounts to multiplication by $c$ on $M^c(A)$:
 \begin{equation}\label{2.4}
    M^c(A) \xrightarrow{\mu_c} A\otimes A^{c-1}\xrightarrow{\lambda_c} M^c(A),\hspace{2cm} \mu_c\lambda_c=c
\end{equation}
see \cite[Section 3]{alexandroustohr}. Finally, for $c\geq 2$ there is a map $\theta_c:M^c(A)\rightarrow L^c(A)$ given by
\begin{equation}\label{2.5}
    [a_1,\dots,a_c]\mapsto \frac{1}{c}\left( \sum_{\sigma
}\left[ a_{1},a_{\pi \left( 2\right) },...,a_{\pi \left( c\right) }%
\right] -\sum_{\tau }\left[ a_{2},a_{\tau \left( 1\right) },...,a_{\tau
\left( c\right) }\right] \right) ,
\end{equation}
where the sums run over all permutations $\sigma$ of $\{2,3,...,n\}$ and
all permutations $\tau $ of $\{1,3,...,n\},$ respectively. Although we work over $\mathbb{Z}$, the factor $1/c$ in \eqref{2.5} makes sense as the expression on the right hand side can be written as a Lie polynomial with integer coefficients(see \cite[Section 2]{bryantkovacsstohr}). This Lie polynomial has been calculated  in \cite[Proposition 7.3]{bryantstohr00}, but since it is rather involved, we prefer to use the compact form \eqref{2.5} in what follows.
The composite of $\theta_c$ and the natural projection $\eta_c$ as in \eqref{2.2} amounts to multiplication by $(c-2)!$ on $M^c(A)$:
\begin{equation}\label{2.6}
    M^c(A) \xrightarrow{\theta_c} L^c(A)\xrightarrow{\eta_c} M^c(A),\hspace{2cm} \theta_c\eta_c=(c-2)!
\end{equation}
(see \cite[Section 2]{bryantkovacsstohr}).

 Now suppose that $A$ carries the structure of a module for the  polynomial ring $U=\mathbb{Z}[X]$ where $X$ is a finite set of variables. Then all the objects introduced in this section such as   Lie powers, symmetric powers etc. will be regarded as $U$-modules under the derivation action. For example, for $x\in X$, $a_i\in A$,
 \begin{equation*}
   [a_1, a_2,  \dots, a_c]x=\sum_{i=1}^{c}[a_1, \dots , a_ix,\dots, a_c],
\end{equation*}
Note that all the maps introduced in this section are compatible with the derivation action, that is, all these maps are, in fact, $\mathbb{Z}[X]$-module homomorphisms. This will be used in what follows without further reference being given. The ring of integers $\mathbb{Z}$ will be regarded as a trivial $U$-module.

\section{Lie powers of free modules}

In this section we retain the notation introduced in Section 2, but now we assume throughout that $A$ is a \emph{free} $U$-module for the polynomial ring $U=\mathbb{Z}(X)$. All homology groups in this section will be over the ground ring $U$. For brevity, if $W$ is a $U$-module, the homology groups $H_k(U,W)={\rm Tor}^U_k(W,\mathbb{Z})$, $k\geq 0,$  will be written as $H_k(W)$.

It is well known (see, for example, \cite[Lemma 5.2]{mansuroglustohr}) that if $A$ is a free $U$-module then both $T^c(A)$ and $A\otimes A^{c-1}$ with $c\geq2$ are also free $U$-modules under the derivation action.

\begin{lemma}
Let $A$ be a free $U$-module, $c\geq 2$. Then
\begin{enumeratei}
\item the tensor products $L_c(A)\otimes_U\mathbb{Z}$ and $M_c(A)\otimes_U\mathbb{Z}$ are direct sums of a free abelian group and a torsion group of exponent dividing $c$,
\item for $k\geq 1$ the homology groups $H_k(L_c(A))$ and $H_k(M_c(A))$ are torsion groups of exponent dividing $c$.
\end{enumeratei}
\end{lemma}
\begin{proof}
By applying the homology functor to the maps in \eqref{2.1} we get that $H_k(\nu_c\rho_c)$ is multiplication by $c$ on $H_k(L_c(A))$ for all $k\geq 0$. Since $A$ is a free $U$-module, $H_0(T_c(A))$ is free abelian and $H_k(T_c(A))=0$ for $k\geq 1$. Then, for any $u\in H_k(L^cA)$ with $k>0$,
\begin{equation*}
    cu=uH_k(\nu_c\rho_c)=(uH_k(\nu_c))H_k(\rho_c)=0H_k(\rho_c)=0.
\end{equation*}

The same holds if $k=0$ and $u\in {\rm Ker}(H_0(\nu_c))$. Hence multiplication by $c$ annihilates both the homology groups $H_k(L_c(A))$ for $k\geq 1$ and the kernel of the homomorphism $H_0(\nu_c)$. The image of this homomorphism is contained in the free abelian group $H_0(T_c(A))$, and hence itself free abelian. The results (i) and (ii) for $L_c(A)$ follow. The proof of (i) and (ii) for $M_c(A)$ are obtained by a similar argument using the maps in \eqref{2.4} instead of those in \eqref{2.1}. \end{proof}

Now we consider Lie powers of prime degree. Let $p$ be a prime.  By applying the homology functor to the short exact sequence \eqref{2.2} we obtain the  exact sequence
\begin{equation}\label{3.1}
  \cdots  \rightarrow H_1(M^p(A))\rightarrow B^p(A)\otimes_U\mathbb{Z}\rightarrow L^p(A)\otimes_U\mathbb{Z}\xrightarrow{ \eta_p\otimes 1} M^p(A)\otimes_U\mathbb{Z}\mapsto 0.
\end{equation}
By Lemma 3.1., $L^p(A)\otimes_U\mathbb{Z}$ is a direct sum of a free abelian group and an elementary abelian $p$-group, and $H_1(M^p(A))$ is an elementary abelian $p$-group. Note that this does not exclude the possibility that these torsion subgroups are trivial. Now the exact sequence \eqref{3.1} yields that $B^p(A)\otimes_U\mathbb{Z}$ is a direct sum of a free abelian group and a (possibly trivial) $p$-group of exponent dividing $p^2$. In fact, we will show that this group has actually no torsion, in other words, we will prove the following result.
\begin{lemma}
Let $A$ be a free $U$-module and  $p$  a prime. Then the tensor product $B^p(A)\otimes_U\mathbb{Z}$ is a free abelian group.
\end{lemma}
\begin{proof} Since we already know that  $B^p(A)\otimes_U\mathbb{Z}$ is a direct sum of a free abelian group and a $p$-group of finite exponent, it is sufficient to show that no non-zero element in $B^p(A)\otimes_U\mathbb{Z}$ is annihilated by $p$. We use reduction modulo $p$, that is the short exact sequence
\begin{equation*}
    0\rightarrow B^p(A)\xrightarrow{p}B^p(A)\rightarrow B^p(A)\otimes\mathbb{Z}_p\rightarrow 0
\end{equation*}
which, in its turn, gives rise to the exact sequence
\begin{equation}\label{3.2}
    \cdots\rightarrow H_1(B^p(A)\otimes\mathbb{Z}_p)\rightarrow B^p(A)\otimes_U\mathbb{Z}\xrightarrow{p}B^p(A)\otimes_U\mathbb{Z}\rightarrow B^p(A)\otimes_U\mathbb{Z}_p\rightarrow 0.
\end{equation}
The Lemma will be proved once we show that the homology group on the left is zero, and this will certainly follow if we can verify that $B^p(A)\otimes\mathbb{Z}_p$, regarded as a module for the polynomial ring $\mathbb{Z}_p[X]$, is projective. The proof of this fact will be given in the next section (see Corollary 4.2). This will then complete the proof of Lemma 3.2.
\end{proof}
Now we have all the ingredients in place to prove the main result of this section. Recall the homomorphism $\theta_c: M^c(A) \rightarrow L^c(A)$ defined by \eqref{2.5}.

\begin{proposition} Let $A$ be a free $U$-module and $p$ a prime. Then
 the torsion subgroups of $L^p(A)\otimes_U\mathbb{Z}$ and $M^p(A)\otimes_U\mathbb{Z}$ are isomorphic, and the homomorphism  $\theta_p\otimes 1$ maps the latter isomorphically onto the former.
\end{proposition}

\begin{proof}
By Lemma 3.1(i) both  $L^p(A)\otimes_U\mathbb{Z}$  and  $M^p(A)\otimes_U\mathbb{Z}$ are direct sums of a free abelian group and an elementary abelian $p$-group. Consider the maps in \eqref{2.6}. By trivializing the $U$-action we obtain homomorphisms
\begin{equation*}
    M^p(A)\otimes_U\mathbb{Z} \xrightarrow{\theta_p\otimes 1} L^p(A)\otimes_U\mathbb{Z}  \xrightarrow{\eta_p\otimes 1} M^p(A)\otimes_U\mathbb{Z} ,\hspace{1cm} \theta_k\eta_k\otimes1=(p-2)!
\end{equation*}
So the restriction of the composite $\theta_k\eta_k\otimes 1$ to the torsion subgroup of the tensor product $M^p(A)\otimes_U\mathbb{Z}$, an elementary abelian $p$-group, is multiplication by $(p-2)!$, that is, it is an isomorphism. It follows that the homomorphism $\theta_p\otimes 1$ maps the torsion subgroup of $M^p(A)\otimes_U\mathbb{Z}$ isomorphically into the torsion subgroup of $L^p(A)\otimes_U\mathbb{Z}$. To prove the proposition, we need to verify  that this map is also surjective, that is, the homomorphism $\theta_k\otimes 1$ maps the torsion subgroup of  $M^p(A)\otimes_U\mathbb{Z}$ isomorphically \emph{onto} the torsion subgroup of $L^p(A)\otimes_U\mathbb{Z}$. But this is true since otherwise the restriction of $\eta_p\otimes 1$ to the torsion subgroup of $L^p(A)\otimes_U\mathbb{Z}$ would have a non-trivial kernel. This, however, is not the case, as follows from the exactness of \eqref{3.1}. Since $B^p(A)\otimes_U\mathbb{Z}$ is free abelian by Lemma 3.2, and $H_1(M^p(A))$ is torsion by Lemma 3.1(ii), there cannot be any torsion elements in the kernel of  $\eta_p\otimes 1$. This proves the proposition.
\end{proof}
In the next section we fill in the gap left in the proof of Lemma 3.2.

\section{The degree $p$ Lie power in characteristic $p$}

In this section $V$ denotes a vector space over a field $K$ of prime characteristic $p$. Moreover, we will assume that $V$ is a module for the polynomial ring $K[X]$   where $X$ is a finite set of indeterminates. Otherwise we will use all the notation introduced in Section 2, in particular, $L^p(V)$, $M^p(V)$ and $T^p(V)$ are the $p$th Lie, metabelian Lie, and tensor powers of $V$, respectively,  $B^p(V)=L^p(V)\cap L(V)''$ is the kernel of the natural projection  $L^p(V)\rightarrow M^p(V)$, and all of these will be regarded as $K[X]$-modules under the derivation action. Recall that $L^p(V)$ may be regarded as a submodule of the tensor power $T^p(V)$, and hence $B^p(V)$ is also a submodule of $T^p(V)$.

\begin{lemma}
The submodule $B^p(V)$ is a direct summand of the $K[X]$-module $T^p(V)$.
\end{lemma}

This result is essentially proved as Theorem 3.1 in \cite{bryantstohr}, except that there the module $V$ is assumed to be finite-dimensional. In what follows we reproduce the proof from \cite{bryantstohr} with some minor amendments necessary to accommodate infinite dimensional modules.

\begin{proof}
For each $r \ge 1$ choose a basis ${\mathcal{B}}^{(r)}$ of $L^r(V)$ and
let ${\mathcal{B}} = \bigcup_r {\mathcal{B}}^{(r)}$. Thus $\mathcal{B}$ is a basis of
$L(V)$. For $b \in {\mathcal{B}}$, let ${\rm deg}(b)$ denote the degree of
$b$, that is, ${\rm deg}(b) =r$ for $b \in {\mathcal{B}}^{(r)}$. Order
$\mathcal{B}$ in any way subject to $b < b'$ whenever ${\rm deg}(b)
< {\rm deg}(b')$. By the Poincar\'e--Birkhoff--Witt Theorem,
$T^p(V)$ has a basis ${\mathcal{C}}$ consisting of all products of the form
$b_1\otimes b_2\otimes\cdots  \otimes b_k$ with $b_1,\dots,b_k \in {\mathcal{B}}$,
$b_1 \le b_2 \le \cdots \le b_k$ and
${\rm deg}(b_1) + \cdots + {\rm deg}(b_k) = p$.
More specifically, any basis element $c \in {\mathcal{C}}$ has the form
\begin{equation}\label{4.1}
c = b_1^{(1)}\otimes\cdots\otimes b_{k_1}^{(1)}\otimes b_1^{(2)}\otimes\cdots \otimes b_{k_2}^{(2)}\otimes
\cdots\otimes b_1^{(p)}\otimes\cdots\otimes b_{k_p}^{(p)},
\end{equation}
where $k_1,\dots,k_p$ are non-negative integers such that
$k_1 + 2k_2 + \cdots +pk_p = p$ and where, for $i=1,\dots,p$, we have
$b_1^{(i)}, \dots, b_{k_i}^{(i)} \in {\mathcal{B}}^{(i)}$
and $b_1^{(i)} \le \cdots \le b_{k_i}^{(i)}$.
We call the $p$-tuple $(k_1,\dots,k_p)$ the {\it type\/}
of $c$ and denote it by ${\rm type}(c)$. Let $\Omega$ denote the set of
all such types. We order $\Omega$ (lexicographically) by
$(k_1,\dots,k_p) > (k_1',\dots,k_p')$ if for some $j \in \{1,\dots,p\}$
we have $k_i = k_i'$ for all $i<j$ but $k_j > k_j'$. Using this
ordering, write $\Omega = \{\omega_1,\omega_2,\dots,\omega_m\}$
where $\omega_1>\omega_2>\cdots >\omega_m$. Thus $\omega_1 =
(p,0,\dots,0)$ and $\omega_m = (0,\dots,0,1)$.

For $i=1,\dots,m$, define
${\mathcal{C}}_i = \{c \in {\mathcal{C}} : {\rm type}(c) = \omega_i\}$
and let $X_i$ denote the subspace of $T^p(V)$ spanned by
${\mathcal{C}}_i \cup {\mathcal{C}}_{i+1} \cup \cdots \cup {\mathcal{C}}_m$.
Also, write $X_{m+1} = 0$. Thus
$$T^p(V) = X_1 > X_2 > \cdots > X_m > X_{m+1} = 0.$$
Note that $X_i/X_{i+1}$ has basis ${\mathcal{C}}_i$ modulo $X_{i+1}$.
Furthermore, ${\mathcal{C}}_1$ consists of all products $b_1^{(1)}\otimes
b_2^{(1)}\otimes\cdots\otimes b_p^{(1)}$ with $b_1^{(1)},\dots,b_p^{(1)} \in
{\mathcal{B}}^{(1)}$ and $b_1^{(1)} \le \cdots \le b_p^{(1)}$.
Also, ${\mathcal{C}}_m = {\mathcal{B}}^{(p)}$ and $X_m = L^p(V)$.

Let $i \in \{1,\dots,m\}$ where $\omega_i = (k_1,\dots,k_p)$.
For $c \in {\mathcal{C}}_i$ written as in \eqref{4.1} it is well known
and easy to verify that the value of $c$ modulo $X_{i+1}$
is unchanged by any permutation of the factors $b_1^{(1)},
\dots,b_{k_p}^{(p)}$. In particular, for all
$\pi_1 \in {\rm Sym}(k_1)$, \dots, $\pi_p \in {\rm Sym}(k_p)$,
we have
\begin{equation}\label{4.2}
b_{\pi_1(1)}^{(1)}\otimes \cdots\otimes b_{\pi_1(k_1)}^{(1)}\otimes \cdots \otimes b_{\pi_p(1)}^{(p)}\otimes
\cdots\otimes b_{\pi_p(k_p)}^{(p)} + X_{i+1} = c + X_{i+1}.
\end{equation}
It follows easily that $X_i$ is a $K[X]$-submodule of $T^p(V)$.
For $c$ written as before let $\overline{c} \in
S^{k_1}(L^1(V)) \otimes \cdots \otimes S^{k_p}(L^p(V))$
be defined by
$$\overline{c} = (b_1^{(1)} \circ \cdots \circ b_{k_1}^{(1)})
\otimes (b_1^{(2)} \circ \cdots \circ b_{k_2}^{(2)}) \otimes
\cdots \otimes (b_1^{(p)} \circ \cdots \circ b_{k_p}^{(p)}).$$
Clearly $\{\overline{c} : c \in {\mathcal{C}}_i \}$ is a basis of
$S^{k_1}(L^1(V)) \otimes \cdots \otimes S^{k_p}(L^p(V))$. Furthermore,
it follows easily from \eqref{4.2} that the linear map given by
$c + X_{i+1} \mapsto \overline{c}$ is
a $K[X]$-module isomorphism from $X_i/X_{i+1}$ to
$S^{k_1}(L^1(V)) \otimes \cdots \otimes S^{k_p}(L^p(V))$. Thus
\begin{equation}\label{4.3}
X_i/X_{i+1} \cong S^{k_1}(L^1(V)) \otimes \cdots \otimes S^{k_p}(L^p(V)).
\end{equation}

Suppose that $i \in \{2,\dots,m-1\}$ where $\omega_i = (k_1,\dots,k_p)$.
Thus $k_1,\dots,k_p < p$ (and, in fact, $k_p = 0$).
Let $\sigma_i : S^{k_1}(L^1(V)) \otimes \cdots
\otimes S^{k_p}(L^p(V)) \to T^p(V)$ be the linear map defined on
the basis $\{\overline{c} : c \in {\mathcal{C}}_i\}$ by
$$\sigma_i(\overline{c}) = \frac{1}{k_1!\cdots k_p!}
\sum_{\pi_1 \in {\rm Sym}(k_1), \dots, \atop \pi_p \in {\rm Sym}(k_p)}
b_{\pi_1(1)}^{(1)}\otimes \cdots\otimes b_{\pi_1(k_1)}^{(1)}\otimes \cdots \otimes b_{\pi_p(1)}^{(p)}\otimes
\cdots\otimes b_{\pi_p(k_p)}^{(p)},$$
where $c$ is written as in \eqref{4.1}
It is straightforward to verify that $\sigma_i$ is a $K[X]$-module homomorphism.
It follows easily from \eqref{4.2} that $\sigma_i(\overline{c}) \in X_i$ and,
indeed, $\sigma_i(\overline{c}) + X_{i+1} = c + X_{i+1}$ for all $c \in
{\mathcal{C}}_i$. Thus the map $\sigma_i$ is injective and we have
$X_i = {\rm Im}(\sigma_i) \oplus X_{i+1}$. Since $X_m = L^p(V)$ we
therefore have
$$X_2 = {\rm Im}(\sigma_2) \oplus \cdots \oplus {\rm Im}(\sigma_{m-1})
\oplus L^p(V).$$
Let $W$ be the submodule of $X_2$ given by
\begin{equation}\label{4.4}
W = {\rm Im}(\sigma_2) \oplus \cdots \oplus {\rm Im}(\sigma_{m-1})
\oplus (L(V)'' \cap L^p(V)).
\end{equation}
Consider the maps
\begin{equation*}
    \alpha: T^p(V)\rightarrow V\otimes V^{p-1}, \hspace{2cm} \beta:V\otimes V^{p-1}\rightarrow T^p(V)
\end{equation*}
defined by
\begin{equation*}
    a_1\otimes a_2\otimes\dots\otimes a_p\mapsto a_1\otimes (a_2\circ\dots\circ a_p)
\end{equation*}
and
\begin{equation*}
     a_1\otimes (a_2\circ\dots\circ a_p)\mapsto \frac{1}{(p-1)!}
\sum_{\pi}a_1\otimes a_{\pi(2)}\otimes\cdots\otimes a_{\pi(p)},
\end{equation*}
where $a_i\in V$ and $\pi$ runs over all permutations of $\{2,3,\dots,c\}$. Then $\alpha$ is surjective and the composite $\beta\alpha$ is the identity map on $V\otimes V^{p-1}$. Hence we have a direct decomposition
\begin{equation*}
    T^p(V)={\rm Ker}(\alpha)\oplus {\rm Im}(\beta)\cong {\rm Ker}(\alpha)\oplus\;( V\otimes V^{p-1}).
\end{equation*}
We claim that ${\rm Ker}(\alpha)=W$. It is easily seen that $W$ is contained in $\rm Ker(\alpha)$.  To verify that we have actually equality, note that the elements
\begin{equation}\label{4.5}
    b_1^{(1)}\otimes
b_2^{(1)}\otimes\cdots\otimes b_p^{(1)}\;\;\;\; \text{with}\;\; b_1^{(1)},\dots,b_p^{(1)} \in
{\mathcal{B}}^{(1)}\;\; \text{and}\;\; b_1^{(1)} \le \cdots \le b_p^{(1)}
\end{equation}
form a basis of $T^p(V)$ modulo $X_2$. Furthermore, the Lie products
\begin{equation}\label{4.6}
     [ b_1^{(1)},
b_2^{(1)},\ldots, b_p^{(1)}]\;\;\;\; \text{with}\;\; b_1^{(1)},\dots,b_p^{(1)} \in
{\mathcal{B}}^{(1)}\;\; \text{and}\;\; b_1^{(1)}>b_2^{(1)} \le \cdots \le b_p^{(1)}
\end{equation}
form a basis of $L^p(V)$ modulo $L^p(V)\cap L(V)''$ (see Section 2). It follows that the elements \eqref{4.5} together with the elements \eqref{4.6} form a basis of $T^c(V)$ modulo $W$. Moreover, the images of these elements under the map $\alpha$ form a basis of $ V\otimes V^{p-1}$. Indeed, we have
\begin{equation}\label{4.7}
   ( b_1^{(1)}\otimes
b_2^{(1)}\otimes\cdots\otimes b_p^{(1)})\alpha=b_1^{(1)}\otimes
(b_2^{(1)}\circ\cdots\circ b_p^{(1)})
\end{equation}
and
\begin{equation}\label{4.8}
     ([ b_1^{(1)},
b_2^{(1)},\ldots, b_p^{(1)}])\alpha=b_1^{(1)}\otimes
(b_2^{(1)}\circ\cdots\circ b_p^{(1)})-\;b_2^{(1)}\otimes
(b_1^{(1)}\circ\cdots\circ b_p^{(1)}).
\end{equation}
This can easily be seen from the short exact sequence \eqref{2.3} (with $V$ instead of $A$). Indeed, the elements \eqref{4.7} are mapped by $\kappa_p$ one-to-one onto the canonical basis of $V^p$, and the elements \eqref{4.8} are precisely the images of the canonical basis elements of $M^p(V)$ under the map $\mu_p$. Consequently, we have the desired equality ${\rm Ker(\alpha)}=W$, and so $T^p(V) = W \oplus {\rm Im}(\beta)$. By \eqref{4.4},
$L^p(V)\cap L(V)''$ is a direct summand of $W$. Thus $L^p(V)\cap L(V)''$ is a direct summand of $T^p(V)$ and we have Lemma 4.1.
\end{proof}
Now the result we need to complete the proof of Proposition 3.3 follows easily.
\begin{corollary}
If $V$ is a free $K[X]$-module, then $B^p(V)$ is a projective $K[X]$-module.
\end{corollary}
\begin{proof}
If $V$ is a free $K[X]$-module, then the tensor power $T^p(V)$ is also a free $K[X]$-module, see \cite[Lemma 5.2]{mansuroglustohr}). Since $B^p(V)$ is a direct summand of  $T^p(V)$, it is projective.
\end{proof}

 \section{The main result}
 In this Section $L$ denotes a free Lie ring of finite rank with free generating set $X$. Our aim is to determine the torsion subgroup of the additive group of the quotient \eqref{1.1}. In view of the
short exact sequence
$$
0 \rightarrow (L')^c/[(L')^c,L] \rightarrow L/[(L')^c,L] \rightarrow L/(L')^c
\rightarrow 0,
$$
this is a free central extension of the free (nilpotent-of-class-c-1)-by-abelian Lie ring
$ L/(L')^c$.  The additive structure of the latter is well-understood. Its underlying abelian group is free abelian \cite{bokut'}. Consequently, any torsion elements must be contained in the central quotient $(L')^c/[(L')^c,L]$, and it is this quotient we will focus on from now on. By the Shirshov-Witt Theorem, the derived ideal $L'$ is itself a free Lie ring, namely, the free Lie ring on $L'/L''$: $L'=L(L'/L'')$. This is a graded Lie ring  and its degree $c$ homogeneous component $L^c(L'/L'')$ is isomorphic to the lower central quotient
 \begin{equation}\label{5.1}
  L^c(L'/L'')\cong (L')^c/(L')^{c+1}.
 \end{equation}
 The adjoint representation induces on these lower central quotients the structure of an $L/L'$-module, and hence of a module for its universal envelope $U=U(L/L')$. The latter may be identified with the polynomial ring on $X$: $U=\mathbb{Z}[X]$. Thus \eqref{5.1} is actually  a $U$-module isomorphism.
 In view of the canonical isomorphism
 \begin{equation*}
   ((L')^c/(L')^{c+1}\otimes_U\mathbb{Z}\cong (L')^c/[(L')^c,L],
 \end{equation*}
   trivializing the $U$-action on both sides of \eqref{4.1} gives an isomorphism
 \begin{equation}\label{5.2}
 (L')^c/[(L')^c,L]\cong  L^c(L'/L'')\otimes_U\mathbb{Z}.
 \end{equation}
 We will exploit this isomorphism to investigate the additive structure of the quotient on the left hand side by examining the tensor product on the right hand side.

 Suppose that $L$ has rank $2$ and, say, $X=\{x,y\}$. Then $L'/L''$ is a free cyclic module over the polynomial ring  $U=\mathbb{Z}[x,y]$ with free generator $[y,x]$, see \cite[Proof of Theorem 6.1]{alexandroustohr}. If $c$ is a prime, say $c=p$, then Proposition 3.3 applies to the tensor product on the right hand side of \eqref{5.2}. Hence this tensor product is a direct sum of a free abelian group and an elementary abelian $p$-group. Moreover, the torsion subgroup is the image in $L^p(L'/L'')\otimes_U\mathbb{Z}$ of the torsion subgroup of $M^p(L'/L'')\otimes_U\mathbb{Z}$ under the map $\theta_p\otimes 1$. A complete description of the latter is given in \cite[Corollary 6.2]{alexandroustohr}. The elements
 \begin{equation*}
  [[u,y],[u,x],\underbrace{u,\dots,u}_{p-2}]\otimes 1
  \end{equation*}
  where $u=[y,x,\underbrace{x,\dots,x}_s,\underbrace{y,\dots,y}_t]$ with $s,t\geqslant 0$ form a basis of this torsion subgroup as a $\mathbb{Z}_p$-module. Applying $\theta_p\otimes 1$ to such a basis element gives
  \begin{equation*}
    \frac{(p-2)!}{p}\sum_{i=0}^{p-1}([[u,y],\underbrace{u,\dots,u}_i,[u,x],\underbrace{u,\dots,u}_{p-2-i}]-[[u,x],\underbrace{u,\dots,u,}_i,[u,y],\underbrace{u,\dots,u}_{p-2-i}])\otimes 1.
  \end{equation*}
Since this is an element of order $p$ we can drop the factor of $(p-2)!$ in the statement of our main result, which summarizes the discussion in this final section.
\begin{theorem}
Let $L$ be a free Lie ring of rank $2$ with free generators $x$ and $y$, and let $p$ be a prime. Then the quotient $(L')^p/[(L')^p,L]$ is a direct sum of a free abelian group and an elementary abelian $p$-group. Modulo $[(L')^p,L]$ the elements
\begin{equation*}
    \frac{1}{p}\sum_{i=0}^{p-1}([[u,y],\underbrace{u,\dots,u}_i,[u,x],\underbrace{u,\dots,u}_{p-2-i}]-[[u,x],\underbrace{u,\dots,u,}_i,[u,y],\underbrace{u,\dots,u}_{p-2-i}])
  \end{equation*}
  where $u=[y,x,\underbrace{x,\dots,x}_s,\underbrace{y,\dots,y}_t]$ with $s,t\geqslant 0$ form a basis of this torsion subgroup as a $\mathbb{Z}_p$-module.\hfill \qed
\end{theorem}
The legality of the factor $1/p$ in the statement of the theorem is explained in Section 2.

\setlength{\parindent}{0pt}
\end{document}